\documentclass[10pt]{article}
\usepackage{amsmath,amstext, amssymb,amsthm,amsfonts,amscd,euscript,verbatim, t1enc, newlfont, xypic} 
\usepackage{hyperref}
\usepackage{color}
\usepackage{cancel} 
\hypersetup{breaklinks}
\theoremstyle{plain}
\newtheorem*{qu}{Question}
\newtheorem{theo}{Theorem}[subsection]

\newtheorem{theore}{Theorem}[section]
\newtheorem{pr}[theo]{Proposition}

\newtheorem{corol}[theore]{Corollary}
 \newtheorem{lem}[theo]{Lemma}
 \newtheorem{coro}[theo]{Corollary}
  
\theoremstyle{remark}
\newtheorem{rema}[theo]{Remark}

\theoremstyle{definition}
\newtheorem{defi}[theo]{Definition}

\newtheorem{exe}[theo]{Example}

 \newcommand\lan{\langle}
\newcommand\ra{\rangle}
 
\newcommand\ob{^{-1}}

\newcommand\obj{\operatorname{Obj}}

\newcommand\id{\operatorname{id}}

\newcommand\cu{\underline{C}}
\newcommand\du{\underline{D}}
\newcommand\au{\underline{A}}
\newcommand\bu{\underline{B}}

\newcommand\n{\mathbb{N}}
\newcommand\z{{\mathbb{Z}}}
\newcommand\re{{\mathbb{R}}}

\newcommand\zop{{\mathbb{Z}[\frac{1}{p}]}}
\newcommand\q{{\mathbb{Q}}}

\newcommand\p{\mathbb{P}}

\newcommand\codim{\operatorname{codim}}

\newcommand\lam{\Lambda}
\newcommand\al{\alpha}
\newcommand\be{\beta}
\newcommand\gam{\gamma}

\newcommand\de{\delta}

\newcommand\lvect{L\mbox{-}\operatorname{vect}}

\newcommand\ns{\{0\}}

\DeclareMathOperator\imm{\operatorname{Im}}
\DeclareMathOperator\co{\operatorname{Cone}}

\DeclareMathOperator\kar{\operatorname{Kar}}

\newcommand\hw{{\underline{Hw}}}
\newcommand\karw{{\operatorname{Kar}^w_{\max}}}
\newcommand\wkar{\operatorname{Kar}^w_{\min}}

\newcommand\bb{{\mathcal{B}}}
\newcommand\sss{{\mathcal{S}}}
\newcommand\spe{\operatorname{Spec}}
\newcommand\kgl{\operatorname{KGl}}
\newcommand\kglp{\operatorname{KGl}'}
\newcommand\sht{SH}

\newcommand\dk{\mathcal{DK}}
\newcommand\mgl{\operatorname{MGl}}

\newcommand\shmgl{D^{\operatorname{MGl}}}
\newcommand\md{\mathcal{D}}
\newcommand\mdc{\mathcal{D}^c}
\newcommand\oo{{\pmb{1}}}
\newcommand\wchow{{w_{\operatorname{Chow}}}}

\newcommand\chow{\operatorname{Chow}}

\newcommand\dmgm{DM_{gm}}
\newcommand\dm{DM}

\numberwithin{equation}{subsection}

\begin{document}

 \title{On constructing weight structures and extending them to idempotent completions} 
 \author{Mikhail V. Bondarko, Vladimir A. Sosnilo
   \thanks{ 
 Research is supported by the Russian Science Foundation grant no. 16-11-10073.}}\maketitle
\begin{abstract}
In this paper we describe a new method for constructing a weight structure $w$ on a triangulated category $\cu$.

For a given $\cu$ and $w$ it allows us to give a fairly comprehensive (and new) description of triangulated categories containing $
\cu$ as a dense subcategory 
(i.e., of subcategories of the idempotent  completion of $\cu$ that contain $\cu$; we call them {\it  idempotent extensions}  of $\cu$) 
to which $w$ extends. In particular, any bounded above or below $w$ extends to any idempotent extension of $\cu$; however, we illustrate by an example that $w$ does not extend to the  idempotent  completion of $\cu$ in general.

We also describe an application of our results to certain triangulated categories of (relative) motives.
\end{abstract}

\tableofcontents

 \section*{Introduction}

In this article we consider the following questions.

\begin{qu} When can a weight structure $w$ for a triangulated category  $\cu$ be  extended to its idempotent completion $\kar(\cu)$? 

More generally, when can $w$ be  extended to  some full triangulated subcategory $\cu'$ of $\kar(\cu)$ containing $\cu$ (we will call a category $\cu'$ satisfying these conditions an {\it idempotent extension} of $\cu$)?

\end{qu}

We provide the following answer.

\begin{theore}
\begin{enumerate}
\item\label{i1} If an extension of $w$ to an idempotent extension of  $\cu$ exists, then it is unique.

\item\label{i2} If $w$ is  bounded above or bounded below, the extension exists.

\item\label{i3} More generally, if an idempotent extension admits an extended weight structure, then it is contained in the category $\karw(\cu)\subset \kar(\cu)$; see Definition \ref{dpkar} for the notation. Even though    $\karw(\cu)$ may contain sub-idempotent extensions not  admitting  extended weight structures,
$\karw(\cu)$ itself admits an extended weight structure.

\item\label{i4} There exist triangulated categories $\cu$ admitting weight structures such that $\karw(\cu)\not\cong \kar(\cu)$   
 (and so,  $w$ does not extend to $\kar(\cu)$).

\item\label{i5} There exists an idempotent extension of $\cu$ (for any $(\cu,w)$) such that $w$ extends to it and the heart of this extended weight structure is idempotent complete.\footnote{This result is important for \cite{bkw} (and so, to the study of the conservativity of the so-called weight complex functor).} 
\end{enumerate}
\end{theore}

This is a very significant improvement  of the previous state of the art \cite[Proposition 5.2.2]{bws}.

Moreover,  our current arguments are substantially easier than the ones used for the proof of loc. cit. The existence of weight structures statements in the theorem (in its parts \ref{i2}, \ref{i3}, and \ref{i5}) are proved using a somewhat technical (yet  not really difficult) Theorem \ref{tex}. 
As another consequence of 
 the latter theorem we obtain the following result (see \S\ref{smot} for  the definitions, references, and a generalization). 

\begin{corol}\label{cint}
Let $B$ be a Noetherian separated excellent scheme of finite Krull dimension and of exponential characteristic $p$; let $R$ be a commutative unital $\zop$-algebra. Choose some generalized dimension function $\delta$ on separated schemes of finite type over $B$ (see Definition \ref{ddf} below); let $j\in \z$.
Then the  Chow weight structure  (as constructed in \cite{bonivan}; see 
 Proposition \ref{pwchow}(1) below) on the motivic category $\dm^c_{cdh}(B,R)$ (defined in \cite{cdint}) restricts to $\dm^c_{cdh}(B,R)_{\de \le j}$.
\end{corol}

We also prove that  any triangulated category $\cu$ that is densely generated (see \S\ref{snotata}) by its negative (see Definition \ref{dwso}(\ref{id6})) additive subcategory $\bu$  
admits a bounded weight structure whose heart is the  retraction-closure of $\bu$ in $\cu$. This statement generalizes the widely cited Theorem 4.3.2(II) of \cite{bws}. 
Note that the latter theorem (along with the aforementioned Proposition 5.2.2 of ibid.) has found several applications to motives, to representation theory (see \cite{postov}, \cite{koyang}
and \cite{kayang};\footnote{In these papers weight structures were called co-$t$-structures following \cite{konk}.}\ cf. also \cite{plam} and \cite{kellerw}); it was also applied to the mixed Hodge theory in \cite{vologod} and to the study of the stable homotopy category of (topological) spectra in  \cite[\S4.6]{bws} (along with \cite[\S2.4]{bkw}).\footnote{The authors also plan to generalize the corresponding  "topological" results to categories of equivariant spectra.}


Let us  now describe the contents  of the paper.

In \S\ref{sprelim} we introduce some basic (mostly, categorical) notation and recall some of the theory of weight structures. None of the statements in this section are really new.

In \S\ref{smain} we prove the aforementioned general existence of  weight structures results.

In \S\ref{sexamples} we demonstrate by simple examples that an (unbounded) weight structure $w$ for $\cu$ does not necessarily extend to all idempotent extensions of $\cu$. Moreover, the category $\kar(\cu)$ does not have to be equivalent to 
its (essentially) maximal triangulated subcategory $\karw(\cu)$ 
to which $w$ extends\footnote{In contrast, for a triangulated category $\du$ with a $t$-structure, the $t$-structure extends to the idempotent completion $\kar(\du)$ (see Theorem 15 of \cite{idemptstr}).}, and there also can exist idempotent extensions of $\cu$ inside $\karw(\cu)$ such that $w$ does not extend to them.

In \S\ref{smot} we describe some ("relative") motivic applications of 
Theorem \ref{tex} and prove (a generalization of) Corollrary \ref{cint}.

The authors are deeply grateful  to the referee, to prof. Ch. Weibel, and to prof. A. Zvonareva for their very useful comments.

\section{Preliminaries }\label{sprelim}

In \S\ref{snotata} we introduce some 
notation and 
conventions, and recall some results on triangulated categories. 
In \S\ref{ssws}  we 
recall 
some basics on weight structures.

\subsection{Some terminology and a few results 
 on triangulated categories}\label{snotata}
  
For categories $C$ and $D$ we write 
$D\subset C$ if $D$ is a full 
subcategory of $C$. 
    
 For a category $C$ and  $X,Y\in\obj C$ 
we will write $C(X,Y)$ for the set of  $C$-morphisms from  $X$ into $Y$.
We will say that $X$ is  a {\it
retract} of $Y$ if $\id_X$ can be factored through $Y$. Note that if $C$ is triangulated  
then $X$ is a  retract of $Y$ if and only if $X$ is a direct summand.

For a category $C$ the symbol $C^{op}$ will denote its opposite category.

For a subcategory $D\subset C$
we will say that $D$ is {\it retraction-closed} in $C$ if $D$ contains all retracts of its objects in $C$. 
  We will call the
smallest retraction-closed subcategory $\kar_C(D)$ of $C$ containing $D$ (here "$\kar$" is for Karoubi) 
 the {\it retraction-closure} of $D$ in $C$. 
The class $\obj \kar_C(D)$ will also be (abusively) called  the retraction-closure of $D$; so we will say that this class is retraction-closed in $C$.

The {\it  idempotent  completion} $\kar(\bu)$ (no lower index) of an additive
category $\bu$ is the category of "formal images" of idempotents in $\bu$
(so, $\bu$ is embedded into a 
category that is {\it  idempotent  complete}, i.e., any idempotent endomorphism splits in it).

The symbols $\cu$ and $\cu'$ will always denote some triangulated categories.
 We will use the
term {\it exact functor} for a functor of triangulated categories (i.e.,
for a  functor that preserves the structures of triangulated
categories).

A class $D\subset \obj \cu$ will be called {\it extension-closed} if $0\in D$ and for any
distinguished triangle $A\to B\to C$  in $\cu$ we have the following implication: $A,C\in
D\implies B\in D$. In particular, any extension-closed $D$ is strict in $\cu$ (i.e., contains all objects of $\cu$ isomorphic to its elements).

The full subcategory 
of $\cu$ whose object class  is the smallest extension-closed $D\subset \obj \cu$ containing a given $D'\subset \obj \cu$ will be called the {\it extension-closure} of $D'$.  Sometimes we will also abusively use this term for $D$ itself.

Below we will need the following simple fact.

\begin{lem}\label{lretr}
Let $M,N\in \obj \cu$, $n\ge 0$, and assume that $N$ is a retract of $M$. Then 
$N$ belongs to the extension-closure of $\{N[2n]\}\cup\{M[i],\ 0\le  i< 2n\}$ 
\end{lem}
\begin{proof}
Assume that $M\cong N\bigoplus P$. Then the assertion is given by the 
(split) distinguished triangles $M[2j]\to N[2j]\to P[2j+1]$ and $M[2j+1]\to P[2j+1]\to N[2j+2]$ for $0\le j<n$.
\end{proof}

The smallest extension-closed $D\subset \cu$ that is  also closed with respect to retracts and  contains a given $D'\subset \obj \cu$ will be called the {\it envelope} of $D'$.

We will say that  a class $D\subset \obj \cu$ {\it strongly generates} a subcategory $\du\subset \cu$ and write $\du=\lan D\ra_{\cu}$   if $\du$
is the smallest full strict triangulated subcategory of $\cu$ such that $D\subset \obj \du$. Certainly, 
this condition is equivalent to $ \du$ being the extension-closure of $\cup_{j\in \z}D[j]$.

We will say that $D\subset \obj \cu$ {\it densely generates} a subcategory $\du\subset \cu$ whenever $\du$ is smallest  retraction-closed triangulated subcategory of $\cu$ such that $D\subset \obj \du$. Certainly, this condition is equivalent to $\obj \du$ being the envelope of $\cup_{j\in \z}D[j]$.

We will say (following \S1.4 of \cite{thom})  that a full strict triangulated subcategory $\cu$ of a triangulated $\cu'$ is {\it dense} in $\cu'$ if $\kar_{\cu'}\cu=\cu$.
Recall that (according to Theorem 1.5 of \cite{bashli}) the category  $\kar(\cu)$ can be naturally endowed with the structure of a triangulated category so that the  natural embedding functor $\cu\to \kar(\cu)$ is exact. Hence if $\cu$ is a dense subcategory of $\cu'$ then there exists a fully faithful  exact functor $\cu'\to \kar(\cu)$.
Moreover, the subcategory $\cu_1$ of $\cu$ that is strongly generated by some class $D\subset \obj \cu$ is dense in the  subcategory $\cu_2$ of $\cu$ densely generated by $D$.

  
For $X,Y\in \obj \cu$ we will write $X\perp Y$ if $\cu(X,Y)=\ns$.
For $D,E\subset \obj \cu$ we  write $D\perp E$ if $X\perp Y$
 for all $X\in D,\ Y\in E$.
For $D\subset \obj \cu$ the symbol $D^\perp$ will be used to denote the class
$$\{Y\in \obj \cu:\ X\perp Y\ \forall X\in D\}.$$
  Dually, ${}^\perp{}D$ is the class
$\{Y\in \obj \cu:\ Y\perp X\ \forall X\in D\}$. 


In this paper all complexes will be cohomological, i.e., the degree of
all differentials is $+1$. 
We will write $K(\bu)$ for the homotopy category of 
 complexes over an additive category $\bu$. Its full subcategory of
bounded complexes will be denoted by $K^b(\bu)$. 
	
	Since triangulated categories of complexes give examples of weight structures important for the current paper, we recall the following simple statements.
		
		\begin{pr}\label{pkbu}
1.  The full subcategories of $K(\bu)$  corresponding to classes of $\bu$-complexes 
  concentrated in degrees $\ge 0$ and $\le 0$ are idempotent  complete. 

2.   The  classes of bounded $\bu$-complexes that are
homotopy equivalent to complexes
 concentrated in degrees $\ge 0$ and $\le 0$ are retraction-closed in $K^b(\bu)$.  

\end{pr}
\begin{proof}
1. This is a part of   \cite[Theorem 3.1]{schnur} (cf. also Proposition 4.2.4 of \cite{sosn}; one should take $F(-)=\coprod_{i\ge 0}-[2i]$  in it).

2. See Remark 6.2.2(1) of \cite{bws}.

\end{proof}

\subsection{Weight structures: basics}\label{ssws}

Let us recall the definition of the main notion of this paper.

\begin{defi}\label{dwstr}

 A couple of subclasses $\cu_{w\le 0}$ and $ \cu_{w\ge 0}\subset\obj \cu$ 
will be said to define a {\it weight structure} $w$ for a triangulated category  $\cu$ if 
they  satisfy the following conditions.

(i) $\cu_{w\le 0}$ and $\cu_{w\ge 0}$ are 
retraction-closed in $\cu$
(i.e., contain all $\cu$-retracts of their elements).

(ii) {\bf Semi-invariance with respect to translations.}

$\cu_{w\le 0}\subset \cu_{w\le 0}[1]$ and $\cu_{w\ge 0}[1]\subset
\cu_{w\ge 0}$.

(iii) {\bf Orthogonality.}

$\cu_{w\le 0}\perp \cu_{w\ge 0}[1]$.

(iv) {\bf Weight decompositions}.

 For any $M\in\obj \cu$ there
exists a distinguished triangle
$X\to M\to Y
{\to} X[1]$
such that $X\in \cu_{w\le 0} $ and $ Y\in \cu_{w\ge 0}[1]$.
\end{defi}

We will also need the following definitions.

\begin{defi}\label{dwso}

Let $i,j\in \z$.

\begin{enumerate}
\item\label{id1} The full subcategory  $\hw\subset \cu$ whose object class is $\cu_{w=0}=\cu_{w\ge 0}\cap \cu_{w\le 0}$ 
 is called the {\it heart} of  $w$.

\item\label{id2} $\cu_{w\ge i}$ (resp. $\cu_{w\le i}$,  $\cu_{w= i}$) will denote $\cu_{w\ge 0}[i]$ (resp. $\cu_{w\le 0}[i]$,  $\cu_{w= 0}[i]$).

\item\label{id3} $\cu_{[i,j]}$  denotes $\cu_{w\ge i}\cap \cu_{w\le j}$; so, this class  equals $\ns$ if $i>j$.

$\cu^b\subset \cu$ will be the category whose object class is $\cup_{i,j\in \z}\cu_{[i,j]}$.

\item\label{id4}  We will  say that $(\cu,w)$ is {\it  bounded}  if $\cu^b=\cu$ (i.e., if
$\cup_{i\in \z} \cu_{w\le i}=\obj \cu=\cup_{i\in \z} \cu_{w\ge i}$).

Respectively, we will call $\cup_{i\in \z} \cu_{w\le i}$ (resp. $\cup_{i\in \z}\cu_{w\ge
i}$) the class of $w$-{\it bounded above} (resp. $w$-{\it bounded below}) objects; we will say that $w$ is bounded above (resp. bounded below) if all the objects of $\cu$ satisfy this property.

\item\label{id5} Let $\cu$ and $\cu'$ 
be triangulated categories endowed with
weight structures $w$ and
 $w'$, respectively; let $F:\cu\to \cu'$ be an exact functor.

$F$ is said to be  {\it  weight-exact}  (with respect to $(w,w')$) if it maps $\cu_{w\le 0}$ into $\cu'_{w'\le 0}$ and 
 $\cu_{w\ge 0}$ into $\cu'_{w'\ge 0}$.

\item\label{id6} Let $\bu$ be a 
full additive subcategory of a triangulated category $\cu$.

We will say that $\bu$ is {\it negative} (in $\cu$) if
 $\obj \bu\perp (\cup_{i>0}\obj (\bu[i]))$.

\end{enumerate}
\end{defi}

\begin{rema}\label{rstws}

1. A  simple (though rather important) example of a weight structure comes from the stupid
filtration on $K(\bu)$ (or on $K^b(\bu)$,  $K^-(\bu)$, or $K^+(\bu)$) for an arbitrary additive category
 $\bu$. 
In either of these  categories we take
$\cu_{w\le 0}$ (resp. $\cu_{w\ge 0}$) to be the class of objects in $\cu$ 
 that are homotopy equivalent to those complexes in $\cu\subset K(\bu)$ that are concentrated in degrees $\ge 0$ (resp. $\le 0$).  Then weight decompositions of objects are given by stupid filtrations of complexes, and the only non-trivial axiom to check  is that the classes 
$\cu_{w\le 0}$ and $\cu_{w\ge 0}$ are retraction-closed in $\cu$; this fact is immediate from Proposition \ref{pkbu}.
 
 The heart of this {\it stupid} weight structure 
is the retraction closure  of $\bu$
 in 
$\cu$.

2. A weight decomposition (of any $M\in \obj\cu$) is (almost) never canonical. 

Still for  $m\in \z$ some choice of a weight decomposition of $M[-m]$ shifted by $[m]$ is often needed (though in the current paper we will only be concerned with  $m$ equal to $0$ or $-1$). So we choose a distinguished triangle \begin{equation}\label{ewd} w_{\le m}M\to M\to w_{\ge m+1}M \end{equation} 
with some $ w_{\ge m+1}M\in \cu_{w\ge m+1}$ and $ w_{\le m}M\in \cu_{w\le m}$. 
 We will   use this notation below (though $w_{\ge m+1}M$ and $ w_{\le m}M$ are not canonically determined by $M$). 
 
3. In the current paper we use the ``homological convention'' for weight structures; 
it was previously used in \cite{wild}, 
 \cite{hebpo}, \cite{brelmot},  \cite{bmm}, \cite{bonivan}, \cite{bkw},  and in \cite{binters},
  whereas in 
\cite{bws} 
 the ``cohomological convention'' was used. In the latter convention 
the roles of $\cu_{w\le 0}$ and $\cu_{w\ge 0}$ are interchanged, i.e., one
considers   $\cu^{w\le 0}=\cu_{w\ge 0}$ and $\cu^{w\ge 0}=\cu_{w\le 0}$. So,  a
complex $X\in \obj K(\au)$ whose only non-zero term is the fifth one (i.e.,
$X^5\neq 0$) has weight $-5$ in the homological convention, and has weight $5$
in the cohomological convention. Thus the conventions differ by ``signs of
weights''; 
 respectively, $K(\au)_{[i,j]}$ is the 
retraction closure in $K(A)$ of the class of complexes concentrated in degrees $[-j,-i]$.

4. 
Actually, in \cite{bws} both "halves" of $w$ were required to be  additive.  Yet the proof of Proposition 1.3.3(1,2) of ibid.  (that is essentially Proposition \ref{pbw}(\ref{iort}) below)  did not use  this additional assumption, whereas that statement easily yields the additivity of  $\cu_{w\le 0}$ and $\cu_{w\ge 0}$ (since it implies Proposition \ref{pbw}(\ref{iext})). Moreover, Definition 2.4 of \cite{konk} (where weight structures were defined independently from \cite{bws}) did not require  $\cu_{w\le 0}$ and $\cu_{w\ge 0}$ to be additive also.

 5. The orthogonality axiom in Definition \ref{dwstr} immediately yields that $\hw$ is negative in $\cu$.
A certain converse to this statement is given by Corollary \ref{cneg} below.

\end{rema}

Let us recall some basic  properties of weight structures. 
Starting from this moment we will assume that $\cu$ is (a triangulated category) endowed with a (fixed) weight structure $w$.

\begin{pr} \label{pbw}
Let  
 $M,M'\in \obj \cu$, $g\in \cu(M,M')$. 

\begin{enumerate}

\item \label{idual}
The axiomatics of weight structures is self-dual, i.e., for $\du=\cu^{op}$
(so $\obj\du=\obj\cu$) there exists the (opposite)  weight
structure $w^{op}$ for which $\du_{w^{op}\le 0}=\cu_{w\ge 0}$ and
$\du_{w^{op}\ge 0}=\cu_{w\le 0}$.


 \item\label{iort}
 $\cu_{w\ge 0}=(\cu_{w\le -1})^{\perp}$ and $\cu_{w\le 0}={}^{\perp} \cu_{w\ge 1}$.

\item\label{iext} 
 $\cu_{w\le 0}$, $\cu_{w\ge 0}$, and $\cu_{w=0}$
are (additive and) extension-closed. 

\item\label{iextb}
The full subcategory $\cu^+$ (resp. $\cu^-$) of  $\cu$ 
whose objects are the $w$-bounded below (resp. bounded above) objects of $\cu$ is a  retraction-closed triangulated subcategory of $\cu$.

\item\label{iextcub} $ \cu^b$ is the extension-closure of $\cup_{i\in\z}\cu_{w=i}$ in $\cu$. 

\item\label{iextlr} If $w$ is bounded then $ \cu_{w\le 0}$ (resp.  $ \cu_{w\ge 0}$) is the extension-closure of $\cup_{i\le 0}\cu_{w=i}$ (resp. of  $\cup_{i\ge 0}\cu_{w=i}$) in $\cu$.

\item\label{iuni} Let  $v$ be another weight structure for $\cu$; assume   $\cu_{w\le 0}\subset \cu_{v\le 0}$ and $\cu_{w\ge 0}\subset \cu_{v\ge 0}$. Then $w=v$ (i.e., the inclusions are equalities).

\end{enumerate}
\end{pr}
\begin{proof}

All of these assertions  were proved  in \cite{bws} 
 (pay attention to Remark \ref{rstws}(3) above!). 

\end{proof}

\begin{rema}\label{restr}
 For $\cu$ endowed with a weight structure $w$ and a triangulated subcategory $\du\subset \cu$ we will say that $w$ {\it restricts} to $\du$ whenever the couple $(\cu_{w\le 0}\cap \obj \du,\cu_{w\ge 0}\cap \obj \du)$ gives a weight structure $w_{\du}$ for $\du$. Part \ref{iort} of our proposition 
easily implies that $w$ restricts to $\du$ if and only if the embedding $\du\to \cu$ is weight-exact with respect to a certain weight structure for $\du$; if this weight structure exists then it is equal to $w_{\du}$  as described by the previous sentence.

\end{rema}

\section{Main results}\label{smain}

This is the central section of the paper. 

In \S\ref{snewgen} we prove our (new) general results on the existence of weight structures.
In \S\ref{spkar} we apply these statements to extending weight structures to {\it  idempotent extensions} of $\cu$.

\subsection
{The general existence of weight structures results}\label{snewgen}


\begin{theo}\label{tex}
Let $\cu'$ be a triangulated category.
Assume given two classes $\cu'_-$ and $\cu'_+$ of objects of $\cu'$  satisfying  the axioms (ii) [Translation Semi-Invariance] and (iii) [Orthogonality]   of Definition \ref{dwstr} (for $\cu_{w\le 0}$ and $ \cu_{w\ge 0}$, respectively). Let us call a $\cu'$-distinguished triangle $X\to M\to Y[1]$ a {\it pre-weight decomposition} of $M$ if $X$ belongs to the 
 envelope  $\cu'_{w'\le 0}$ of  $\cu'_-$ and $Y$ belongs to the 
envelope $\cu'_{w'\ge 0}$ of  $\cu'_+$.

I. Then the following statements are valid.

1.  The class of objects possessing pre-weight decompositions is extension-closed (in $\cu'$). Moreover, if $M$ and $N\in \obj\cu'$ possess pre-weight decompositions then any $\cu'$-extension of $M$ by $N$ possesses a pre-weight decomposition whose components are some extensions of the corresponding components of pre-weight decompositions of $M$ and of $N$, respectively (cf.  \cite[Lemma 1.5.4]{bws}).

2. $\cu'_{w'\le 0}\subset \cu'_{w'\le 0}[1]$ and $\cu'_{w'\ge 0}[1]\subset
\cu'_{w'\ge 0}$.

3. $\cu'_{w'\le 0}\perp \cu'_{w'\ge 0}[1]$.

 4. Let $C'$ be a subclass of $ \obj \cu'$ such that  $\cu'$ is the extension-closure  of  $C'$ and any element of  $C'$ possesses a pre-weight decomposition.
Then  the couple $(\cu'_{w'\le 0},\cu'_{w'\ge 0})$ gives a weight structure $w'$  for $\cu'$.

5. Assume (in addition to the assumptions of the previous assertion) that $C'=\cup_{i\in \z}C[i]$ for some $C\subset \obj \cu'$ and that 
for any $c\in C$ there exists $i_c\in \z$ such that $c[i_c]\in \cu'_{w'\ge 0}$ (resp.  $c[i_c]\in \cu'_{w'\le 0}$). Then this $w'$ is bounded below (resp. bounded above).

II. 
Suppose that a class $C''\subset \obj \cu'$ satisfies the following conditions: $\cu'$ is  densely generated by $C''$ (see \S\ref{snotata}),
 pre-weight decompositions exist for $c[i]$ whenever $c\in C''$ and $i\in \z$, 
and for any $c\in C''$ there exists $i_c\in \z$ such that  
 $c[i_c]\in \cu'_{w'\le 0}$ (cf. assertion I.5).\footnote{See the proof of Corollary \ref{cmot} below for an example of  $\cu'_-$, $\cu'_+$, and $C''$ (for a certain $\cu'$).}  

Then the couple $(\cu'_{w'\le 0},\cu'_{w'\ge 0})$ is a weight structure  for $\cu'$ in this case also; this weight structure $w'$ is bounded below.

Moreover, $w'$ is also bounded above if we assume in addition that for any $c\in C''$ there exists $i'_c\in \z$ such that  $c[i'_c]\in \cu'_{w'\le 0}$.

III. Assume that $N\in \cu'_{w'\le 0}$ is a retract of some $M\in \obj \cu'$ and let $X\to M\to Y[1]\to X[1]$ be a pre-weight decomposition (of $M$). Then the following statements are valid.

1. $N$ is a retract of $X$.

2.  Suppose that $N'\in \cu'_{w'\ge 0}$ is a retract of some $M'\in \obj \cu'$ and let $A'\to M'[1]\to B'[1]\to A'[1]$ be a pre-weight decomposition of $M'[1]$. Then $N'$ is a retract of $B'$.

3. Let $A\to X[1]\to B[1]\to A[1]$ be  a pre-weight decomposition of $X[1]$. Then $B\in  \cu'_{w'\le 0}\cap  \cu'_{w'\ge 0}$. 
Moreover, if $N $ also belongs to $\cu'_{w'\ge 0}$ then $N$ is a retract of $B$.

\end{theo}
\begin{proof}
I.1. See Remark 1.5.5(1) of \cite{bws}.

2,3. Obvious from the corresponding properties of   $(\cu'_-,\cu'_+)$.

4. We use an easy and more or less standard argument; it was first applied to weight structures in the proof of \cite[Theorem 4.3.2(II.1)]{bws}.

Certainly $\cu'_{w'\le 0}$ and $\cu'_{w'\ge 0}$ are retraction-closed in $\cu'$. 
Axioms (ii) and (iii) of weight structures are fulfilled for  $(\cu'_{w'\le 0},\cu'_{w'\ge 0})$ according to the previous assertions.

We only have to verify the existence of weight decompositions (for all objects of $\cu'$). This  statement is an immediate consequence of assertion I.1. 


5. 
Immediate from Proposition \ref{pbw}(\ref{iextb}).

II. We take $C'=\cup_{i\in \z}C''[i]\cup \cu'_{w\ge 0}[1]$. 
Certainly, all elements of $C'$ possess pre-weight decompositions. According to assertion I.4, the couple $(\cu'_{w'\le 0},\cu'_{w'\ge 0})$ gives a weight structure for $\cu'$ if  $ \cu'$  equals the 
extension-closure of 
$C'$; 
so we verify the latter fact.

Denote by $\cu''$ the triangulated subcategory of $\cu'$ strongly generated by $C''$. According to assertion I.1,  any $c\in \obj \cu''$  possesses a pre-weight decomposition, and there (also) exists  $i_c\in \z$ such that  $c[i_c]\in \cu'_{w'\le 0}$. Now,  any object of $\cu'$ is a retract of an  object of $\cu''$; hence it also satisfies the latter property. Applying Lemma \ref{lretr}  we easily deduce that $\cu'$ equals the envelope of $\obj \cu''\cup  \cu'_{w\ge 0}[1]$; hence it also equals the $\cu'$-envelope of $C'$.

Lastly, the  boundedness below of $w'$ along with the "moreover" part of the assertion follows immediately from Proposition  \ref{pbw}(\ref{iextb}). 

III.1. 
Recall that $N$ being a retract of $M$ means that $\id_N$ can be factored through $M$.
Next, we have $N\perp Y[1]$; hence the corresponding morphism from $N$ into $M$ can be factored through $X$. 

2. This assertion can be easily seen to be 
the categorical dual of the previous one (cf. Proposition \ref{pbw}(\ref{idual})).

3. Recall that $\cu'_{w'\le 0}$ 
 is extension-closed in $\cu'$. Hence the distinguished triangle $X\to B\to A$ gives $B\in \cu'_{w'\le 0}$.
Next, $B\in \cu'_{w'\ge 0}$ by the definition of a pre-weight decomposition.

Lastly, $N$ is a retract of $X$ according to assertion III.1; hence the "moreover" part of this assertion follows immediately from the previous assertion.

\end{proof}

Now we describe an easy application of  our theorem (along with previous results).

\begin{coro}\label{cneg}

Let $\bu$ be an (additive) negative subcategory (see Definition \ref{dwso}(\ref{id6})) of a triangulated category $\cu'$ such that 
$\cu'$ is densely generated by $\obj \bu$. 

 Then the following statements are valid.

1. The envelopes $\cu'_{w'\le 0}$ and $\cu'_{w'\ge 0}$ of the classes $\cup_{i\le 0} \obj \bu[i]$ and $\cup_{i\ge 0} \obj \bu[i]$, respectively, give a weight structure on $\cu'$. 

2. The heart of this weight structure $w'$ equals  $\kar_{\cu'}(\bu)$. 

3. $\obj \kar_{\cu'}(\bu)$  strongly generates $\cu'$. Moreover, $\cu'_{w'\le 0}$ (resp.  $\cu'_{w'\ge 0}$) is the extension-closure of $\cup_{i\le 0}\obj \kar_{\cu'}(\bu)[i]$ (resp. of $\cup_{i\ge 0}\obj \kar_{\cu'}(\bu)[i]$). 

4. $w'$ is the only weight structure for $\cu'$ whose heart contains $\bu$.

\end{coro}
\begin{proof}

1. It suffices to note that the classes $C'_-=\cup_{i\le 0}\obj \bu[i]$, $C'_+=\cup_{i\ge 0}\obj \bu[i]$, and $C''=\obj \bu$   satisfy the conditions of  Theorem \ref{tex}(II). Indeed, $C'_-\perp C'_+[1]$ since $B$ is negative, and all the other conditions are obvious.

2. Denote  by $\cu$ the triangulated subcategory of $\cu'$ that is strongly generated by $\bu$. Then part I.4 of our theorem immediately implies that $w'$ restricts to $\cu$ (in the sense of Remark \ref{restr}). Denote the corresponding weight structure on $\cu$ by $w$.\footnote{Its existence is precisely 
 Theorem 4.3.2(II.1) of \cite{bws}.} 

Now,  applying (the "moreover" statement in) part I.1 of our theorem to $(\cu,w)$ we obtain that for any $M\in \obj \cu$ there exists a choice of $X=w_{\le 0}M$ belonging to the extension-closure of $\cup_{i\le 0}\obj (\bu[i])$. Applying  the same 
part of the theorem  to  $X$ we obtain the existence of a choice of  $w_{\ge 0}X$ belonging to $\obj \bu$.

Next, any object $N$ of $\cu'$ is a retract of some object $M$ of $\cu$. 
Taking an arbitrary $N\in \cu'_{w'=0}$ and considering  the corresponding $w_{\ge 0}X$ as described above we obtain that $N$ is a retract of $w_{\ge 0}X\in \obj\bu$ according to part III.3 of the theorem. Hence  $\hw'$ equals $\kar_{\cu'}(\bu)$.

3. Since $w'$ is bounded, $\cu'_{w'=0}$ strongly generates $\cu'$ according to Proposition \ref{pbw}(\ref{iextcub}). Combining this with  assertion 2 we obtain the first part of our assertion, whereas  Proposition \ref{pbw}(\ref{iextlr}) gives the second part. 

4. Let $v$ be a weight structure for $\cu'$ whose heart $\underline{Hv}$ contains $\bu$. Then $\underline{Hv}$ certainly contains $\kar_{\cu'}(\bu)$.  Now, the classes $\cu'_{v\le 0}$ and  $\cu'_{v\ge 0}$ contain the extension-closures of $\cup_{i\le 0}(\cu'_{v=0}[i])$ and of  $\cup_{i\ge 0}(\cu'_{v=0}[i])$, respectively. Applying the previous assertion we obtain $\cu'_{w'\le 0}\subset \cu'_{v\le 0}$ and $\cu'_{w'\ge 0}\subset \cu_{v\ge 0}$. Thus our uniqueness assertion follows from Proposition \ref{pbw}(\ref{iuni}).
\end{proof}

\begin{rema}\label{rneg}
For $\cu'$ as above being idempotent  complete our  corollary gives   Proposition 5.2.2 of \cite{bws}.
 So we obtain a new proof of loc.\ cit.\ that only relies on \S1 of ibid. (and so, it is somewhat easier than the original one). 

The general case of Corollary \ref{cneg} 
is completely new.

\end{rema}

\subsection{On extending weight structures to  idempotent extensions}\label{spkar}


\begin{defi}\label{dpkar}
1. We will call a triangulated category $\cu'$ an {\it idempotent extension} of $\cu$ if  it contains $\cu$ and there exists a fully faithful exact functor  $\cu'\to \kar(\cu)$.\footnote{The latter assumption is certainly equivalent to any of the  following conditions: any object of $\cu'$ is a retract of some object of $\cu$; $\cu$ is dense (see \S\ref{snotata})  in $\cu'$.}\

2. We will say that a weight structure $w$  {\it extends} to an idempotent extension $\cu'$ of $\cu$ 
 whenever there  exists a weight structure 
 $w'$ for $\cu'$ such that the embedding $\cu\to \cu'$ is weight-exact.  In this case we will call $w'$ an {\it extension} of $w$. 

3. We will say that a triangulated category $\cu'$ endowed with a weight structure $w'$ is {\it weight-Karoubian} if $\hw'$ is idempotent  complete.

4. We will call  a weight-Karoubian category $(\cu',w')$ a   {\it weight-Karoubian extension} of $(\cu,w)$ if $\cu'$ is an  idempotent extension of $\cu$ and $w'$ is an extension of $w$ to $\cu'$.

5.  The (triangulated) category  $\lan \obj \cu\cup \obj \kar(\hw)\ra_{\kar(\cu)}$ will be denoted by $\wkar(\cu)$, and 
the category $\lan \obj \kar(\cu^-)\cup \obj \kar(\cu^+) \ra_{\kar(\cu)}$  (see Proposition \ref{pbw}(\ref{iextb}))  will be denoted  by $\karw(\cu)$. 

\end{defi}

Now we 
study those  idempotent extensions   of $\cu$ such that $w$ extends to them.

\begin{theo}\label{tpkar}
Let $\cu'$ be an  idempotent extension  of $\cu$.

I.1. Assume that $w'$ is an extension of $w$ to $\cu'$.
Then $\cu'_{w\le 0}$ (resp.  $\cu'_{w'\ge 0}$,   $\cu'_{w'= 0}$)  is the retraction-closure of  $\cu_{w\le 0}$ (resp.  $\cu_{w\ge 0}$,   $\cu_{w= 0}$) in $\cu'$. 

2.  An extension of $w$ to $\cu'$ exists if and only if  $\cu'$ is strongly generated by $\obj \cu\cup C_1\cup C_2$ for some class $C_1$  of  retracts of objects of $\cu^+$ and some class $C_2$ of  $\cu'$-retracts of objects of $\cu^-$.

II.1.  An extension $w'$ of $w$ is bounded  below (resp. above) if and only if $w$ is.

2. Assume that $w$ is either bounded below or bounded above. Then $w$ extends to any  idempotent extension of $\cu$.

III.1. 
The categories $\wkar(\cu)\subset \karw(\cu)$ when equipped with the unique extensions of $w$ to them are weight-Karoubian extensions of $\cu$.

2. If $\cu'$ is a weight-Karoubian extension of $\cu$ then $\hw'$ is equivalent to the  idempotent  completion of $\hw$.

3. If $w$ extends to $\cu'$ then there exists a   fully faithful exact functor from $\cu'$ into $\karw(\cu)$; this functor is weight-exact with respect to the corresponding ({\it extended}) weight structures.

4. If $\cu'$ is weight-Karoubian then there exists a fully faithful weight-exact functor $\wkar(\cu)\to \cu'$.

\end{theo}
\begin{proof}
I.1. Since $\cu'_{w'\le 0}$, $\cu'_{w'\ge 0}$, and  $\cu'_{w'= 0}$ are retraction-closed  (in $\cu'$), these classes do contain the retraction closures in question.

The proof of the converse implication is 
similar to 
 the proof Corollary \ref{cneg}(2). 
Let an element $N$ of $ \cu'_{w'\ge 0}$ (resp. of  $\cu'_{w'\le 0}$,  $\cu'_{w'= 0}$) be a retract of $M\in \obj \cu$.
Note now that any   $w$-decomposition of  an object of  $\cu$ is also a $w'$-decomposition. Applying Theorem \ref{tex}(III)  
we obtain that $N$ is a retract of any choice of $X=w_{\le 0}M$ (resp. of $w_{\ge 0}M$,  $w_{\ge 0}X$), whereas these three objects belong to $\cu_{w\le 0}$, $\cu_{w\ge 0}$, and  $\cu_{w= 0}$, respectively.

2. Assume that an extension of $w$ to $\cu'$ exists. Then  any object $M$ of $\cu'$ possesses a weight decomposition with respect to $w'$.  Applying assertion I.1, we obtain that this triangle gives a presentation  of $M$ as an extension of an object $M_1$ of $ \kar_{\cu'}(\cu^+)$ by an object $M_2$ of $\kar_{\cu'}(\cu^-)$.  Thus 
one can take $C_1$ to be the class of all $M_1$ obtained this way, and $C_2$ to be the class of all $M_2$.

To verify the converse implication, for $\cu'$ being strongly generated by $\obj \cu\cup C_1\cup C_2$  we should check that the $\cu'$-retraction closures  $\cu'_{w\ge 0}$ and  $\cu'_{w'\le 0}$ of the classes   $\cu_{w\ge 0}$ and  $\cu_{w\le 0}$, respectively, give a weight structure for $\cu'$. For this purpose we apply Theorem \ref{tex}(I.4) for  $\cu'_-= \cu'_{w'\le 0}$ and $\cu'_+=\cu'_{w'\le 0}$. According to this theorem, it suffices to verify that any element of   $\obj \cu\cup (\cup_{i\in \z} C_1[i]) \cup (\cup_{i\in \z} C_2[i]) $ possesses a pre-weight decomposition. 
Certainly, any object of $\cu$ possesses a  pre-weight decomposition inside $\cu$. Hence it suffices to verify the existence of pre-weight decompositions 
for elements of $\cup_{i\in \z} C_1[i]$ (since dualization would yield the same assertion for $\cup_{i\in \z} C_2[i]$; cf. Proposition \ref{pbw}(\ref{idual})).

Thus it suffices to verify the following: for any $j\in \z$ and all pairs $(M,N)$, where $M\in \cu_{w\ge j}$ and $N$ is a $\cu'$-retract of $M$, there exists a pre-weight decomposition of $N$.
This fact is certainly true if $j>0$. In the general case we choose $n\ge 0$ such that $j+2n>0$ and 
recall that $N$ belongs to the extension-closure of $\{N[2n]\}\cup\{M[i],\ 0\le  i< 2n\}$ (see Lemma \ref{lretr}). It remains to apply Theorem \ref{tex}(I.1).

II.1. Certainly, if all objects of $\cu'$ are $w'$-bounded below (resp. above) then all objects of $\cu$ are  $w'$-bounded below (resp. above); hence 
they are $w$-bounded below (resp. above) also.

The converse implication is  immediate from  assertion I.1. 

2. Immediate from 
assertion I.2.

III.1. The weight structure $w$ extends to $\wkar(\cu)$ and to $\karw(\cu)$ according to  assertion I.2; these categories are weight-Karoubian  according to  assertion I.1. Lastly, the existence of weight decompositions in $\cu$ certainly implies that $\wkar(\cu)\subset \karw(\cu)$.

2. Immediate from assertion I.1.

3. The existence of a   fully faithful exact functor  $F:\cu'\to \karw(\cu)$ is immediate from assertion I.2. The functor $F$ is weight-exact according to assertion I.1.

4. 
Certainly, if a  weight-Karoubian extension $\cu'$ of $\cu$  is a strict subcategory of $\kar(\cu)$ then it contains $\kar_{\kar(\cu)}\hw$; this implies the  existence of a full embedding  $\wkar(\cu)\to \cu'$. 
This functor is weight-exact according to assertion I.1. 

\end{proof}

\begin{rema}\label{rpkar}
1. In particular, there exists at most one extension of $w$ to $\cu'$ (so, it may be called "the" extension of $w$ to $\cu'$); its heart can be embedded into the  idempotent  completion of $\hw$.

2. So, any $(\cu,w)$ possesses a weight-Karoubian extension. This fact is important for \cite{bkw}.

3. Certainly, any  idempotent  complete triangulated category with a weight structure is weight-Karoubian, but the converse 
 fails (for unbounded weight structures).  In particular, the categories $\wkar(\cu)$ and $\karw(\cu)$ can be distinct from $\kar(\cu)$
 (see the example in \S\ref{ssosn}). 

4. Obviously, part I.2 of the theorem can be reformulated as follows: $w$ extends to $\cu'$ if and only  if $\cu'$ is strongly generated by $\obj \kar_{\cu'}\cu^-\cup \obj \kar_{\cu'}\cu^+$. 

Assume now that the category  $\cu$ is essentially small; then its  idempotent  completion $\du=\kar(\cu)$ is also essentially small. Next, the categories $\cu'$, $\lan\obj \cu\cup \obj \kar_{\cu'}\cu^- \ra_{\cu'}$,	and $\lan\obj \cu\cup \obj \kar_{\cu'}\cu^+ \ra_{\cu'}$ may be assumed to be dense   in $\du$ for any idempotent extension $\cu'$ of $\cu$. 
	
	Now recall that the Grothendieck group $K_0(\du)$ is defined as follows: it is the abelian group  whose
generators are isomorphism classes of objects of $\du$, and such that for any $\du$-distinguished triangle $X\to Y\to Z$ the relation $[Y]=[X]+[Z]$ on the classes is fulfilled. Furthermore, sending a subgroup $H$ of $K_0(\du)$ into the full subcategory of $\du$ whose objects are characterized by the condition $[M]\in H$ one obtains a one-to-one correspondence between the set of subgroups of   $K_0(\du)$  and  the set of  (all) dense subcategories of $\du$; see Theorem 2.1 of \cite{thom}.

	Thus the subcategories $\lan\obj \cu\cup \obj \kar(\cu^-) \ra_{\du}$ and $\lan\obj \cu\cup \obj \kar(\cu^+) \ra_{\du}$ of $\du$ correspond to certain subgroups $K^-$ and $K^+$ of $K_0(\du)$, and one can easily  check that $w$ extends to $\cu'$ if and only  if 
	for the group $G=\imm (K_0(\cu')\to K_0(\du))$ we have $(G\cap K^-)+(G\cap K^+)=G$. 
	
The authors suspect that this criterion is rather  difficult to apply in general. Note however that these 	Grothendieck group observations have inspired the 
  example described in \S\ref{ssspkar} below.

\end{rema}

\section{Some (counter)examples}\label{sexamples}

By Theorem \ref{tpkar}(II.2), any bounded above (or bounded below) weight structure $w$ on $\cu$
 extends to any idempotent extension of $\cu$. 
  In this section we demonstrate that this statement (along with two of its natural implications) 	fails for a general $w$. 

\subsection
{The category ${\protect\operatorname{Kar}}^w_{\max}({\protect\underline{C}})$  may be strictly smaller than $ {\protect \operatorname{Kar}}({\protect\underline{C}})$}
\label{ssosn}

Certainly, if $\cu^+$ and $\cu^-$ are  idempotent  complete then $\cu\cong \karw(\cu)$. Now we construct an example of this situation with $\cu$ not being  idempotent  complete; it certainly follows that $\karw(\cu)$ is not equivalent to $\kar(\cu)$ in this situation.

Consider the unbounded homotopy category $\cu = K(\au)$ (note that $K(\au)$ doesn't have infinite coproducts  if $\au$ does not, and in particular, is not necessarily idempotent complete), where $\au$ is an additive category with $K_{-1}(\au) \neq 0$; here we endow $\au$ with the trivial 
structure of an exact category and define the groups $K_{*}(\au) $ using Definition 8 of \cite{schlicht}.
Note that for this purpose one can take $\au$ to be the category of finitely generated projective modules over a (commutative) 
  ring $R$ such that $K_{-1}(R)\neq 0$ (see Theorem 5 of ibid.); 
 rings satisfying this condition are well known to exist.  

 Indeed, one can take the affine nodal curve $C= \operatorname{Spec}(R)$ for $R=\mathbb{Q}[x,y]/(y^2 - x^3 - x^2)$. 
 Then the (cartesian) 
  abstract blow-up square (see \cite[\S0]{cdhkth})
 $$\xymatrix{
\operatorname{Spec}(\mathbb{Q}[t]/(t-1)(t+1)) \ar[d] \ar[r]& \operatorname{Spec}(\mathbb{Q}[x,y]/(x,y)) \ar[d]^{}\\
\operatorname{Spec}(\mathbb{Q}[t]) \ar[r]^{x\mapsto t^2-1, y\mapsto t^3-t}              & C}$$ 
yields the exact sequence of cdh-cohomology 
$$
0\to \operatorname{H}_{cdh}^0(C,\mathbb{Z}) \to \operatorname{H}_{cdh}^0(pt,\mathbb{Z})\oplus \operatorname{H}_{cdh}^0(\mathbb{A}^1,\mathbb{Z}) \to \operatorname{H}_{cdh}^0(pt \sqcup pt,\mathbb{Z}) \to \operatorname{H}_{cdh}^1(C,\mathbb{Z}) \to 0
$$
Note that 
 there is an isomorphism of functors $\operatorname{H}_{cdh}^0(-,\mathbb{Z}) \cong  \mathbb{Z}^{\operatorname{comp}(-)} $ 
 where $\operatorname{comp}(X)$ is the set of connected components of a $k$-variety $X$. 
Hence by \cite[Theorem 0.2]{cdhkth} we have $K_{-1}(R)=K_{-1}(C) \cong \operatorname{H}_{cdh}^1(C,\mathbb{Z}) \cong \mathbb{Z}$. 

Next, Corollary 6 of  \cite{schlicht} 
  implies that in this case 
$K(\au)$ is not  idempotent  complete.\footnote{
	Theorem 5, Corollary 6, and Definition 8 in the published version of this paper correspond to Theorem 7.1, Corollary 8.2, and Definition 5.4 in the  K-theory archives preprint version, respectively.}\

Now take $w$ to be  the stupid weight structure for $\cu$ (see Remark \ref{rstws}(1)). Then the categories $\cu^+$ and $\cu^-$ are  idempotent  complete according to 
Proposition \ref{pkbu}(1). 
Thus  $\karw(\cu)$ is equivalent to $\cu$, whereas $\kar(\cu)$ is not, and we obtain the desired example. 


Lastly, applying Theorem \ref{tpkar}(III.3)  we conclude that a weight structure on a triangulated category does not necessarily extend to its  idempotent  completion.

This example also demonstrates that there exist rather "natural" triangulated categories that are not  idempotent  complete.

\subsection{An idempotent extension inside
${\protect\operatorname{Kar}}^w_{\max}({\protect\underline{C}})$
such that $w$ does not extend to it}\label{ssspkar}

Now we construct an example of $(\cu,w)$ and an idempotent extension  $\cu'$ of $\cu$ such that $\cu\cong \wkar{\cu}\subset$ $\cu'\subset \karw(\cu)=\kar(\cu)$, but  
$w$ does not extend to $\cu'$. Certainly, $w$ will not be bounded either above or below (cf. Theorem \ref{tpkar}(II.2)).  

Let $L$  be an arbitrary (fixed) field; denote by $\lvect$ the category of finite dimensional $L$-vector spaces. 

We start from describing our candidate for $\kar(\cu)$; it will be a certain full triangulated subcategory $\du$ of $ K(\lvect)$ (yet it will be convenient for us not to assume that $\du$ is strict in  $ K(\lvect)$, i.e., $\du$ will not be closed with respect to  $ K(\lvect)$-isomorphisms). We will write $M=(M^i)$ if the $L$-vector spaces $M^i$ are the terms of the complex $M$.

The objects $\du$ will be those  $M=(M^i)\in \obj K(\lvect)$ such that the dimensions of $M^i$ are bounded (by some constant depending on $M$). Obviously, $\du$ is a triangulated subcategory of $ K(\lvect)$, and it 
 contains the "standard" cone of any  $\du$-morphism (recall that $\du$  is not strict in $ K(\lvect)$). 
Note also that any 
$M\in \obj \du$ is 
 isomorphic to $M'\in \obj \du$ such that all the differentials of $M'$ are zero. In particular, it follows that $\du$ is idempotent complete. Moreover, the stupid weight structure for $ K(\lvect)$ certainly restricts to $\du$.

For any $M\in \obj \du$ consider the following sequences: 
$$a^j_M=\sum_{0\le i \le j}(-1)^i\dim_L(M^i)\text{ and } b^j_M=\sum_{0\le i \le j}(-1)^i\dim_L(M^{-i}),$$
where $j$ runs through non-negative integers. 
Note that if $M$ is a zero object of $\du$ then these sequences are bounded.

This fact implies the following one: if for $M\in \obj \du$ there exists a real number $\al_M$ such that the sequence $a^j_M-\al_M\cdot j$, $j\in \n$,  is bounded
(resp. $\be_M\in \re$ such that  $b^j_M-\be_M\cdot j$  is bounded) then for any $M'\in \obj \du$ that is isomorphic to $M$ the sequence $a^j_{M'}-\al_M\cdot j$
(resp. $b^j_{M'}-\be_{M}\cdot j$) is bounded also. Thus any real number $\gamma$ defines (obviously)  non-empty subsets $\du_\gamma^+$ and $\du_\gamma^-$  of $\obj \du$ characterized by the conditions $\al_M=\gamma$ and  $\be_M=\gamma$, respectively, and these sets are closed with respect to $\du$-isomorphisms.

Obviously, $\du_\gamma^+[1]=\du_{\,-\gamma}^+$ and $\du_\gamma^-[1]=\du_{\,-\gamma}^-$ for any $\gam\in \re$. Moreover, if $M_1\to M_2\to M_3\to M_1[1]$ is a distinguished triangle in $\du$ then $M_3$ is isomorphic to $\co(M_1\to M_2)\in \obj \du$; hence if there exist 
$\gamma_1,\gamma_2\in \re$ such that $M_i\in \du_{\gamma_i}^+$ (resp. $M_i\in \du_{\gamma_i}^-$)  for $i=1,2$ then $M_3$ belongs to $ \du_{\gamma_2-\gamma_1}^+$ (resp.  to $ \du_{\gamma_2-\gamma_1}^-$). Furthermore, all bounded above (resp. bounded below) objects of $\du$ belong to $\du_{0}^+$ (resp. to $\du_{0}^-$).

Now we are able to describe $\cu$ and $\cu'$. We take $\cu$ to be the subcategory of $\du$ whose object set is $\cup_{(l,m) \in \z\times \z} (\du_{l}^+\cap \du_{m}^-)$. The observations above imply that $\cu$ is a triangulated subcategory of $\du$; moreover, the stupid weight structure for $K(\lvect)\supset \du$ obviously restricts to $\cu$. Since $\cu$ contains $K^b(\lvect)$, the heart of this restricted weight structure $w$ is equivalent to $\lvect$; hence $\cu\cong \wkar{\cu}$. 
 Next, any object $M$ of $\du$ is a retract of an object of $\cu$ (easy; recall that we can assume the differentials of $M$ to be zero); hence $\karw(\cu)=\kar(\cu)\cong \du$.

Thus it remains to specify a triangulated subcategory $\cu'$ of $\du$ that contains $\cu$ and such that the stupid weight structure does not restrict to $\cu'$.
For this purpose it obviously suffices to take the object set of $\cu'$ to be equal to 
$\cup_{(l,r)\in \re\times \z} (\du_{l}^+\cap \du_{l+r}^-)$.


\section{A survey of motivic applications of Theorem \ref{tex}}\label{smot}

Now we describe the application of Theorem \ref{tex}(II) to various "relative motivic" categories. 

\subsection{On  relative motives and Chow weight structures for them: a reminder}

We consider some tensor triangulated categories of motives over  schemes that are separated and of finite type over a (fixed) base scheme $B$. 
We always assume that $B$ is Noetherian separated excellent of finite Krull dimension. We will  call schemes that are separated and of finite type over $B$ just $B$-schemes,  and a $B$-morphism is a  morphism between $B$-schemes. 
\footnote{So, for a $B$-scheme $Y$ a $B$-morphism into $Y$ is just a separated morphism of finite type.}\

Our main examples will be certain full subcategories of  triangulated categories of the following types.

\begin{exe}\label{emot}
\begin{enumerate}
\item\label{ie1}  {\it Beilinson motives.} For any  $B$ satisfying the aforementioned conditions one can consider the categories of Beilinson motives over 
  $B$-schemes. Recall that  Beilinson motives is a version of (generalized) Voevodsky motives with rational coefficients; they were one of the main subjects of  \cite{cd} (that heavily relied on \cite{ayoubsix}). 

\item\label{ie2}  {\it  $cdh$-motives.} If we assume in addition that $B$ is a scheme of characteristic $p$ for $p$ being a prime or zero,  then for any $\zop$-algebra $R$ (we set $\zop=\z$ if $p=0$) one can also consider $R$-linear $cdh$-motives $\dm_{cdh}(-,R)$   over  
 $B$-schemes (this is another version of Voevodsky motives that was studied in detail in \cite{cdint}).

\item\label{ie3}  {\it $K$-motives.}  For any $B$   and $Y$ being a $B$-scheme one can consider the $\lam$-linear version of  the homotopy category of modules over the symmetric motivic ring spectrum $\kglp_Y$, where  $\sss$ is a set of primes containing all primes non-invertible on $B$ and $\lam =\z[\sss\ob]$.
This means the following: as in \S13.3 of \cite{cd} (that relied on \cite{rso}) one should consider  a certain Quillen model for the  motivic stable homotopy category $\sht(Y)$, 
take the category of strict left modules over $\kglp_Y$ (that is a certain highly structured ring spectrum  weakly homotopy equivalent to the Voevodsky's $K$-theory spectrum $\kgl_Y$), and "invert the primes in $\sss$" using the corresponding well-known method (see \cite[\S A.2]{kellyth},  \cite[Appendix B]{levconv}, or \cite[Proposition 1.1.1]{bkl}). 
We will  (following ibid.) use the notation $\dk(Y)$ for this category and call its objects $K$-motives.

\item\label{ie4} {\it Cobordism-motives.}  For any $B$ as in example \ref{ie2}, a $B$-scheme $Y,$ any set of primes $\sss$ containing $p$, and $\lam =\z[\sss\ob]$ one can similarly take the $\lam$-linear version of  the category $\shmgl(Y)$ of strict left modules over the Voevodsky's spectrum $\mgl_Y$ (cf. \cite[Example 1.3.1(3)]{bondegl}).

\end{enumerate}
\end{exe}

Actually, 
 any couple $(B,\md)$ that satisfies a certain (rather long) list of 
 properties  is fine for our purposes; cf. \cite[\S3]{binters}. 

For  $Y$ being a $B$-scheme 
  the full tensor triangulated subcategory of compact objects in $\md(Y)$ (for $\md(-)$ being 	
	any of the four aforementioned motivic categories) will be denoted by $\mdc(Y)$ and its tensor unit will be denoted by $\oo_Y$.

All these categories can be endowed with the corresponding Chow weight structures. We will now present one of many equivalent definitions.
To do this we need to first discuss the six Ayoub-Grothendieck operations.

 For any 
$B$-morphism (of $B$-schemes) $f:X\to Y$ there are two pairs of adjoint functors
$$f_!:\mdc(X)\leftrightarrows \mdc(Y):f^! \text{ and } f^*:\mdc(Y)\leftrightarrows \mdc(X):f_*.$$ 
Next, for any $Y$ we have a natural splitting $g_*(\oo_{\p^1(Y)})\cong \oo_Y\bigoplus \oo_Y \lan -1\ra$ induced by the zero section $Y\to \p^1(Y)$, where $g:\p^1(Y)\to Y$ is the canonical projection, and $ \oo_Y \lan -1\ra$ is $\otimes$-invertible (this splitting can be used as the definition of  $ \oo_Y \lan -1\ra$).  For any $n\in \z$ we will write  $-\lan n \ra$  for the tensor product by the $-n$th power of $ \oo_Y \lan -1\ra$ in $\mdc(Y)\subset \md(Y)$; this is a certain version of Tate twist that is denoted by $-(n)[2n]$ in the  Voevodsky's convention introduced in \cite{1}.

\begin{defi}\label{dwchow}
Let $Y$ be a $B$-scheme.

1.  We will write $\mdc(Y)_{\wchow(Y)\ge 0}$ (resp.   $\mdc(Y)_{\wchow(Y)\le 0}$)  for the envelope of $\{f_*(\oo_P)\lan n \ra[i]\}$ (resp. of   $\{f_!(\oo_P)\lan n \ra[-i]\}$) for $f:P\to Y$ running through all $B$-morphisms with regular domain, $n\in \z$, and $i\ge 0$. 

2. The objects of the category 
$$\chow_{\md}(Y)=\kar_{\mdc(Y)}(\{f_*(\oo_P)\lan n \ra: f:P\to Y \text{ proper, } P  \text{ regular, }n\in \z\}$$ 
will be called {\it $\md$-Chow motives over $Y$}.
\end{defi}

\begin{pr}\label{pwchow}
Let $Y$ be a $B$-scheme.

1. Then the couple $(\mdc(Y)_{\wchow(Y)\le 0}, \mdc(Y)_{\wchow(Y)\ge 0}) $ gives a bounded weight structure $\wchow(Y)$  on $\mdc(Y)$ for $\md$ being any of the examples in \ref{emot}.\footnote{Moreover,  this "compact version" of $\wchow(Y)$ naturally extends to an unbounded weight structure on the whole $\md(Y)$; see \cite[\S2.3]{bkl} and \cite[Proposition 1.2.4]{binters}).} 

2. We have $\chow_{\md}(Y)\subset \hw_{\chow(Y)}$.
\end{pr}
\begin{proof}
1. This fact was established in \cite{bonivan} for the example \ref{ie2}. Moreover, the methods of ibid. can actually be used in all  the four 
 examples; see \S2 of \cite{bkl}, Remark 3.4.3 of \cite{binters}, and Remark \ref{rwchow}(2--5) below for more detail. 

2. Recall that $f_*=f_!$ if $f$ is proper; the assertion follows immediately.

\end{proof}

\begin{rema}\label{rwchow}
1. If $Y$ is the spectrum of a perfect field $k$  and $\mdc(Y)=\dmgm(k)$ is the category of geometric Voevodsky motives (with coefficients in any ring),
$f$ is a  smooth morphism,  then the object $f_!f^!(\oo_Y)$ is isomorphic to the Voevodsky motif of the variety $P$ (over $k$); moreover,   $f_!f^!(\oo_Y)\cong f_!(\oo_X)\lan d\ra$ whenever all connected components of $X$ are of dimension $d$. Hence our definition of Chow motives over $Y$  generalizes the description of the subcategory of  Chow motives 
inside 
$\dmgm(k)$ (see \cite{1}); this is why the weight structures considered in this section are called Chow weight structures.

2.  Now we discuss the proof of Proposition \ref{pwchow}(1); this will also explain why we need Theorem \ref{tex} to prove Corollary \ref{cmot} below.

The classes $\mdc(Y)_{\wchow(Y)\le 0}$ and  $ \mdc(Y)_{\wchow(Y)\ge 0} $ are  retraction-closed by construction; the inclusions   $\mdc(Y)_{\wchow(Y)\le 0}\subset \mdc(Y)_{\wchow(Y)\le 0}[1]$ and \linebreak $ \mdc(Y)_{\wchow(Y)\ge 0}[1]\subset  \mdc(Y)_{\wchow(Y)\ge 0}  $  are automatic also. So we obtain axioms (i) and (ii) of Definition \ref{dwstr} for $\wchow(Y)$.

The proof of the orthogonality axiom (iii) is more complicated; yet the arguments used for the proof of \cite[Lemma 1.3.3]{bonivan} are easily seen to work fine in all of our examples \ref{emot}.

3. However, checking the existence of  $\wchow(Y)$-weight decompositions for objects of $\md^c(Y)$ 
is somewhat more complicated. The authors know three methods for proving this statement.

Firstly, one can try to verify that  $\md$-Chow motives strongly generate $\mdc(Y)$. Since the category $\chow_{\md}(Y)$ is  negative (in $\mdc(Y)\subset \md(Y)$; this fact follows from the orthogonality axiom for $\wchow(Y)$ that we have just discussed), this assertion would yield the remaining axiom (iv) (see Corollary \ref{cneg}). One would also obtain  $\chow(Y)=\hw_{\chow(Y)}$. 

Yet to prove that $\chow(Y)$ strongly generates $\mdc(Y)$ one requires some statement on the "abundance" of proper $Y$-schemes that are regular; thus it is a certain resolution of singularities problem.
In the case where $Y=\spe k$, $k$ a characteristic $0$ field, it was proved in \cite[Corollary 3.5.5]{1} (using Hironaka's resolution of singularities) that the subcategory of $\mdc(Y)$  strongly generated by the motives of smooth projective $k$-varieties also contains the motives of all smooth varieties. Thus the category  $\lan \obj \chow(k)\ra_{\dmgm(k)}$ is dense in $\dmgm(k)$. Next, a formal argument (essentially using weight decompositions) was applied in \cite{mymot} to prove that $\lan \obj \chow(k)\ra_{\dmgm(k)}$ actually equals $\dmgm(k)$. This method of proof can be applied to all of our four examples of $\mdc(Y)$ (see Example \ref{emot})   whenever $Y$ is of characteristic $0$ (i.e., if it is a $\spe \q$-scheme); see Theorem 2.4.3 of \cite{bondegl}.
For other $Y$ one needs certain alterations (de Jong's ones for rational coefficients and Gabber's  ones in the general case of our Example \ref{emot}(\ref{ie2}--\ref{ie4})) and somewhat more complicated "formal" arguments. So, if the coefficient ring is not a $\q$-algebra then  our current level of knowledge enables us to prove that $\mdc(Y)$ is  strongly generated by $\md$-Chow motives over $Y$ only under the assumption that $Y$ is 
essentially of finite type over a field; see  \cite{bzp} (cf. also \cite[Proposition 5.5.3]{kellyth}) for the case $Y=\spe k$ and \cite[\S2.3]{bonivan} for the general case. For rational coefficients substantially weaker assumptions on $Y$ are sufficient (cf. \cite[\S2.4]{bondegl}); the corresponding method of constructing $\wchow$ was applied in \cite{hebpo}. It appears  that these assumptions on $Y$ (and $B$) are also sufficient to ensure (more or less) "easily" that the 
 weight structure 
 $\wchow(Y)$ restricts (see Remark \ref{restr}) to the levels of the 
dimension filtration for $\mdc(Y)$ (that we will describe below; see Corollary \ref{cmot}).\footnote{Thus one does not need Theorem \ref{tex} to prove Corollary \ref{cmot} in these cases.}
 However, this argument was never written down in the general case (yet cf.  \cite[\S2.4]{bondegl} for a somewhat related reasoning).

4. One more possible definition of $\wchow(Y)$ for the Example \ref{emot}(\ref{ie1}) (for a "general" $Y$) was given in \cite[\S2.3]{brelmot}; it used stratifications of $Y$ (and it was essentially proven in \cite[Theorem 2.2.1]{bonivan} that this definition is equivalent to our Definition \ref{dwchow}(1)). 
Then one can proceed to prove the existence of weight decompositions using the gluing of weight structures argument described in \cite[\S2.3]{brelmot} and \cite[\S2.1]{bonivan} (this is the second method of checking axiom (iv) of Definition \ref{dwstr} for 
$\wchow(Y)$).\footnote{Recall that for $Z$ being any closed subscheme of $Y$ and $U=Y\setminus Z$ the categories $\md(Y)$, $\md(Z)$, and $\md(U)$ along with the natural functors connecting them  yield a {\em gluing datum}   in the sense  of \S1.4.3 of \cite{bbd}; cf. Proposition 1.1.2(10) of \cite{brelmot}. Furthermore, weight structures can be "glued" in this setting according to  Theorem 8.2.3 of \cite{bws}. One also needs certain "continuity" arguments to "glue $\wchow(Y)$ from the Chow weight structures over points of $Y$".}\ Yet this formal argument  does not yield much information on "weights" and weight decompositions of ("concrete") objects of $\mdc(Y)$. In particular, if one considers (following \cite{1}; see below) a certain 
dimension filtration for $\mdc(Y)$ then the gluing argument does not imply that an object belonging to some level of 
this filtration possesses a weight decomposition inside this level.\footnote{The problem is that some of the functors in the aforementioned gluing datum do not respect this filtration.}\

5. To overcome the latter difficulty the third method  
of studying $\wchow$-decompositions (that is more "explicit" than the second one)  was developed in  \cite[\S3]{bkl}. It uses  quite complicated "geometric" arguments (and relies on \cite[Theorems IX.1.1, II.4.3.2]{illgabb}) 
 and is closely related to  Gabber's arguments applied in \cite[\S XIII]{illgabb} to the study of 
 constructibility for complexes of \'etale sheaves.\footnote{Recall also that a reasoning very similar to Gabber's one was  applied to the study of  constructibility of motives in \cite[\S4.2]{cd}.}\ Unfortunately, the corresponding Theorem 3.4.2 of \cite{bkl} is too complicated to be formulated here. So (following \cite[\S3.4]{binters}) we formulate some of its consequences instead (in Proposition \ref{pwdimf}).

6. Recall that the behaviour of the categories $\mdc(-)$ and of the "weights" of their objects is quite similar to that of mixed $\overline{\mathbb{Q}}_l$-complexes of \'etale sheaves and their weights as studied in \cite{bbd}.\footnote{This observation was treated in \S3 of \cite{brelmot}. Note however that "weights" for  mixed complexes of \'etale sheaves do not correspond to any weight structures; see Remark 2.5.2 of \cite{bmm}. On the other hand, we have the (self-dual) perverse $t$-structure $p_{1/2}$ for mixed complexes of sheaves that "respects weights", whereas the existence of its motivic analogue (essentially suggested by Beilinson; cf. \cite{bmm}) is an extremely hard conjecture (that may be true only for motives with coefficients in a $\q$-algebra).}\

In particular, if $X$ is regular then the object $\oo_X$ is a Chow motif over $X$; so it belongs to $\mdc(X)_{\wchow=0}$ (cf. Theorem 5.3.8 of \cite{bbd}). Next, for any $B$-morphism $f$ the functors $f_*$ and $f_!$  possess certain weight-exactness properties with respect to 
	the corresponding Chow weight structures (see Theorem 2.2.1 of \cite{bonivan}; 	cf.  the 'stabilities' 5.1.14 of \cite{bbd}).  Moreover, the functor $-\lan n \ra$ is weight-exact for any $Y$ and any $n\in \z$. 
	
	These observations "motivate" the description of $\wchow$ given in Definition \ref{dwchow}. 
	
\end{rema}

\subsection{On dimension filtrations  and restrictions of $\wchow_{\md}(-)$ to its levels}

Now it is the time to define the dimension filtration for $\mdc(Y)$. 
One of the problems here is that to obtain a "satisfactory" filtration we need some sort of dimension function $\de$ on $B$-schemes; the reason is that we want some notion of dimension that would satisfy the following property: if $U$ is open dense in $X$ then its ``dimension'' $\de(U)$ should be equal to 
$\de(X)$. So we give the following definition following \cite[Definition 3.1.1]{binters}.

\begin{defi}\label{ddf}
\begin{enumerate}
\item\label{ddf-2}  Let $\de^B$ be a 
function from the set $\bb$ of Zariski points of  $B$ into 
 integers\footnote{In some of the formulations of \cite{binters} it is convenient to assume that the values of $\de^B$ are non-negative; yet this additional restriction is easily seen to be irrelevant for the purposes of the current paper.} 
 that satisfies the following condition: if  $b\in \bb$ and a point $b'\in \bb$ belongs to its closure  then $\de(b)\ge \de(b')+\codim_{b}b'$. 

Then for $y$ being a generic point of a 
 $B$-scheme $Y$ (so, $y$ is the spectrum of a field; certainly, we can assume $Y$ to be connected here) and $b\in \bb$ being the image of $y$ in $B$ we set $\de(y)=\de^B(y)=\de^B(b)+\operatorname{tr.\,deg.}k(y)/k(b)$, where $k(y)$ and $k(b)$ are the corresponding  fields. 

\item\label{ddf-3}  For $Y$ being an $B$-scheme we define $\de(Y)$ as the maximum over points of $Y$ of $\de(y)$. 

\end{enumerate}
\end{defi}

\begin{rema}
 In the case where $B$ is a Jacobson scheme all of whose components are equicodimensional one may take $\de$ to be equal to the Krull dimension function.\footnote{In particular, this statement may be applied in the case where $B$ is of finite type over $\spe \z$ or over $\spe k$; yet the latter case is not really interesting to us due to the reasons described above.}\ More generally, one may take $\de$ to be a "true" dimension function as described in \cite[\S XIV.2]{illgabb}
 (cf. \S1.1 of \cite{bondegl}).  
\end{rema}

Now we fix $\de^B$ and the corresponding $\de$ (till the end of the paper). 
The corresponding dimension filtration on $ \mdc(Y)$ is defined as follows: for any $j\in \z$ we take $\mdc_{\de \le j}(Y)$ to be the 
subcategory of  
 $\mdc(Y)$ that is densely generated (see \S\ref{snotata})
by $\{f_!(\oo_P)\lan \de(P)\ra\}$ for $f:P\to Y$ running through all  $B$-morphisms  with $\de(P)\le j$ and regular $P$.

We 
recall the main ingredient of the proof of Theorem 3.4.2(I) of \cite{binters}.

\begin{pr}\label{pwdimf}
For any $j\in \z$ the category  $\mdc_{\de \le j}(Y)$ contains the objects $f_*(\oo_P)\lan \de(P) \ra$ and $f_!(\oo_P)\lan \de(P) \ra$  whenever  $f:P\to Y$ is a  $B$-morphism with $\de(P)\le j$. 
 
Moreover, if $f_0:P_0\to Y$ is a  $B$-morphism,  $\de(P_0)=j$, then for the object $M=f_{0!}(\oo_{P_0})\lan j\ra$ and any $m\in \z$
there exists a distinguished triangle $L\to M\to R\to L[1]$ with $R$ (resp. $L$) belonging to the envelope of $u_*(\oo_U)\lan \de(U) \ra[m+i+1]$ (resp. of   $u_!(\oo_U)\lan \de(U) \ra[m-i]$) for $u:U\to Y$ running through all $B$-morphisms with regular domain and $\de(U)\le j$, and $i\ge 0$. 
\end{pr}
\begin{proof}
This is an easy consequence of  \cite[Theorem 3.4.2]{bkl}; see also Proposition 3.4.1(2) of \cite{binters} for some more detail.
\end{proof}
 
 Let us now 
prove the main statement of this section 
(that generalizes  Corollary \ref{cint}). We take $Y=B$ in it since this does not affect the generality of the statement.

\begin{coro}\label{cmot} Assume that the couple $(B,\md)$ is of any of the types \ref{ie1}--\ref{ie4} described in  Example \ref{emot}.
Consider the Chow weight structure $\wchow(B)$ on $\mdc(B)$ (see 
Proposition \ref{pwchow}(1)).
Then $\wchow(B)$ restricts (see Remark \ref{restr}) to  $\mdc_{\de \le j}(B)$.

\end{coro}
\begin{proof}
Given the results mentioned above, this is an easy application of  Theorem \ref{tex}(II).

 We set $\cu'=\mdc_{\de \le m}(B)$ and take $\cu'_-=\{u_!(\oo_U)\lan \delta(U)\ra[-s]\}$ and $\cu'_+=\{u_*(\oo_U) \lan \delta(U)\ra[s] \}$   for $u:U\to B$ running through   
 $B$-morphisms  with regular domain and $\de(U)\le j$, 
and $s\ge 0$. These two classes obviously satisfy axiom (ii) of Definition \ref{dwstr}. Moreover, we have $\cu'_-\perp \cu'_+[1]$ since $\cu'_-\subset \mdc(B)_{\wchow(B)\le 0}$ and $\cu'_+\subset \mdc(B)_{\wchow(B)\ge 0}$.

We take $C''$ to be the set of all $f_!(\oo_P)\lan \de(P)\ra$ for $P$ being regular,  $\de(P)\le j$, and $f$ being  a $B$-morphism. 
Then $C''$ densely generates $\cu'$ and we have $C''\subset \cu_-$. According to Proposition \ref{pwdimf}, any element of $C''[i]$ for $i\in \z$ possesses a pre-weight decomposition with respect to the corresponding $\cu'_{w'\le 0}$ and  $\cu'_{w'\ge 0}$. Thus we can apply  Theorem \ref{tex}(II) to conclude the proof.

\end{proof}

\begin{rema}\label{rwdimf}
1.  It certainly follows that  $\wchow(B)$ 
also restricts to the union   $\cup_{j\in \z}\mdc_{\de \le j}(B)$ that will be denoted by $\mdc_{eff}(B)$. The latter fact 
  is certainly (formally) weaker than the existence 
	of all the restrictions to $\mdc_{\de \le j}(B)$; yet the authors do not know any  proof of it that does not rely on \cite{bkl}.

2. We will  call $\mdc_{eff}(B)$ the subcategory of $\delta$-effective objects of $\mdc(B)$. Note here that $\mdc_{eff}(B)$ 
 is the subcategory of  $\mdc(B)$ that is densely generated
 by $\{f_!(\oo_P)\lan \de(P)\ra\}$ for $f:P\to B$ running through all $B$-morphisms.  

This definition originates  from Definition 2.2.1 of \cite{bondegl}; it is also closely related to an earlier definition from  \cite[\S2]{pelaez}. 
 
Recall also that the definition of Voevodsky motives (in \cite{1}) actually "started from" certain effective motivic categories. It seems that this method does not work so nicely for general ("relative") motivic categories;  still our definition of $\mdc_{eff}(B)\subset \mdc(B)$ essentially generalizes the one of ibid. according to  \cite[Example 2.3.13(1)]{bondegl}.

3. Certainly, having the category $\mdc_{eff}(B)$ one can also consider the {\it slice filtration} of $\mdc(B)$ by the subcategories $\mdc_{eff}(B)\lan i \ra$ ($=\mdc_{eff}(B)(i)$) for $i$ running through integers. Note also that the union $\cup_{i\in \z}\mdc_{eff}(B)\lan i\ra$ equals $\mdc(B)$. 
Moreover, Theorem 3.4.2(II)  (along with Remark 3.4.3(1)) of \cite{binters} easily implies that the intersection of $\mdc_{eff}(B)\lan i\ra$ for all $i\in \z$ is zero for $\md$ being as in Example \ref{emot}(\ref{ie1}, \ref{ie2}, or \ref{ie4}). On the other hand, $K$-motives are {\it periodic}, i.e., $\oo_B\lan i\ra \cong\oo_B$ for any $i\in \z$ (and so, all $\dk^c_{eff}(B)\lan i \ra$ are equal to $\dk^c(B)$).


4. These two types of filtrations for the "whole"  $\md(-)$ were the main subject of \cite[\S3]{binters}.

\end{rema}



\end{document}